\documentclass[11pt,reqno]{amsart}

\usepackage{url}
\usepackage{amssymb}
\usepackage{amscd}
\usepackage{amsfonts}
\usepackage{amsmath}
\usepackage{amsthm}
\usepackage{amsgen}
\usepackage{array}

\usepackage{eepic}
\usepackage{gastex}

\newcommand{\sgp}{semi\-group}
\newcommand{\sgps}{semi\-groups}
\newcommand{\fs}{finite \sgp}
\newcommand{\fss}{finite \sgps}
\newcommand{\fb}{finitely based}
\newcommand{\ib}{identity basis}
\newcommand{\nfb}{non\-finitely based}

\newcommand{\is}{involuted semi\-group}
\newcommand{\iss}{involuted semi\-groups}
\newcommand{\fbp}{finite basis problem}
\newcommand{\op}{order preserving}

\DeclareSymbolFont{rsfscript}{OMS}{rsfs}{m}{n}
\DeclareSymbolFontAlphabet{\mathrsfs}{rsfscript}

\theoremstyle{plain}
\newtheorem{theorem}{Theorem}
\newtheorem{lemma}[theorem]{Lemma}
\newtheorem{proposition}[theorem]{Proposition}
\newtheorem{corollary}[theorem]{Corollary}

\theoremstyle{remark}
\newtheorem{question}{Question}
\newtheorem*{remarks}{Remarks}
\newtheorem*{remark}{Remark}

\newenvironment{proof*}{\trivlist\item[\hskip\labelsep{\emph{Proof}.}]
}{\endtrivlist}

\title[The Finite Basis Problem for Kiselman Monoids]{The Finite Basis Problem\\ for Kiselman Monoids}
\author{D. N. Ashikhmin, M. V. Volkov}
\address{(D. N. Ashikhmin, M. V. Volkov) Institute of Mathematics and Computer Science, Ural Federal University, Lenina 51, 620000 Ekaterinburg, Russia} \email{dmitry11ksi@gmail.com, mikhail.volkov@usu.ru}

\author{Wen Ting Zhang}
\address{(Wen Ting Zhang) School of Mathematics and Statistics, Lanzhou University, Lanzhou, Gansu 730000, China; Key Laboratory of Applied Mathematics and Complex Systems, Lanzhou, Gansu 730000, China; Department of Mathematics and Statistics, La Trobe University, VIC 3086, Australia}
\email{zhangwt@lzu.edu.cn}

\thanks{M. V. Volkov acknowledges support from the Presidential Programme ``Leading Scientific Schools of the Russian Federation'', project no.\ 5161.2014.1, the Russian Foundation for Basic Research, project no.\ 14-01-00524, the Ministry of Education and Science of the Russian Federation, project no.\ 1.1999.2014/K, and the Competitiveness Program of Ural Federal University. Wen Ting Zhang acknowledges support from the National Natural Science Foundation of China, project nos.\ 11371177, 11401275, the State Scholarship Fund of China, and ARC Discovery Project DP1094578.}

\begin{document}

\begin{abstract}
In an earlier paper, the second-named author has described the identities holding in the so-called Catalan monoids. Here we extend this description to a certain family of Hecke--Kiselman monoids including the Kiselman monoids $\mathcal{K}_n$. As a consequence, we conclude that the identities of $\mathcal{K}_n$ are nonfinitely based for every $n\ge 4$ and exhibit a finite identity basis for the identities of each of the monoids $\mathcal{K}_2$ and $\mathcal{K}_3$.
\end{abstract}

\maketitle

\section{Kiselman and Hecke--Kiselman monoids}
\label{sec:intro}

Kiselman \cite{Ki02} has found all possible compositions of three natural closure operators playing a distinguished role in convex analysis. Namely, let $E$ be a real topological vector space and let $f$ be any function on $E$ with values in the extended real line $\mathbb{R}\cup\{+\infty,-\infty\}$. The first of the three  operators associates to $f$ its \emph{convex hall} or \emph{largest convex minorant} $c(f)$, defined as the supremum of all its convex minorants. Alternatively, it can be shown that for each $x\in E$, the value of $c(f)$ at the point $x$ can be calculated as follows:
\[
c(f)(x)=\inf\left\{\sum_{i=1}^N\lambda_if(x_i)\ \Bigl|\Bigr.\ N\ge1,\ \lambda_i>0,\ f(x_i)<+\infty,\ \sum_{i=1}^N\lambda_ix_i=x\right\}.
\]
The second operator is that of taking the \emph{largest lower semicontinuous minorant} $\ell(f)$ of the function, defined as the supremum of all its lower semicontinuous minorants with respect to the topology on $E$. It can be verified that for each $x\in E$, the value of $\ell(f)$ at $x$ is given by the formula
\[
\ell(f)(x)=\liminf_{y\to x}f(y).
\]
This corresponds to taking the closure of the \emph{epigraph} of the function $f$ (the set of points lying on or above the graph of $f$) with respect to the Cartesian product of the topology on $E$ and the usual topology on $\mathbb{R}$.

The third operator is the operator $m$, defined as
\[
m(f)(x)=\begin{cases}
f(x)   &\text{if }\ f(y)>-\infty\  \text{ for all }\ y\in E;\\
-\infty&\text{otherwise.}
\end{cases}
\]
The operators $c,\ell,m$ generate a monoid $G(E)$ with composition as multiplication. Kiselman \cite[Theorem~4.1]{Ki02} has shown that the order of $G(E)$ can be 1, 6, 15, 16, 17, or 18, depending on the dimension of $E$ and its topology. In particular, for every normed space $E$ of infinite dimension, $G(E)$ consists of 18 elements and has the following presentation (as a monoid):
\begin{multline}
\label{eq:k3}
G(E) = \langle c, \ell, m \mid c^2 = c,\ \ell^2 =\ell,\ m^2 = m,\\
c\ell c = \ell c \ell = \ell c,\ cmc = mcm = mc,\ \ell m\ell = m\ell m = m\ell\rangle.
\end{multline}

Ganyushkin and Mazorchuk (unpublished) have suggested to consider analogous presentations with arbitrarily many generators. Namely, for each $n\ge2$, they have defined the \emph{Kiselman monoid} $\mathcal{K}_n$ as follows:
\begin{multline}
\label{eq:kn}
\mathcal{K}_n = \langle a_1,a_2,\dots,a_n \mid a_i^2=a_i,\ i=1,\dots, n;\\
a_ia_ja_i=a_ja_ia_j=a_ja_i,\ 1\le i < j \le n\rangle.
\end{multline}
Clearly, the monoid defined by the presentation \eqref{eq:k3} is isomorphic to $\mathcal{K}_3$---the bijection $c\mapsto a_1$, $\ell\mapsto a_2$, $m\mapsto a_3$ extends to an isomorphism. Algebraic properties of the family of monoids $\{\mathcal{K}_n\}_{n=2,3,\dots}$ have been studied in depth by Kudryavtseva and Mazorchuk~\cite{KM09}. Here we only mention that the monoid $\mathcal{K}_n$ is finite for every $n$ \cite[Theorem~3]{KM09} but its exact order (as a function of~$n$) is not yet known.

A further generalization motivated by some connections with representation theory has been suggested by Ganyushkin and Mazorchuk in~\cite{GM11}. Fix an integer $n\ge2$ and take an arbitrary anti-reflexive binary relation $\Theta$ on the set $\{1,2,\dots,n\}$. The \emph{Hecke--Kiselman monoid} $\mathcal{HK}_\Theta$ corresponding to $\Theta$ is the monoid generated by elements $a_1,a_2,\dots,a_n$ subject to the relations
\begin{align*}
&a_i^2=a_i                  &&\text{for each } i=1,\dots,n;\\
&a_{i}a_{j}=a_{j}a_{i}      &&\text{if } (i,j),(j,i)\notin\Theta;\\
&a_ia_ja_i=a_ja_ia_j        &&\text{if } (i,j),(j,i)\in\Theta;\\
&a_ia_ja_i=a_ja_ia_j=a_ja_i &&\text{if } (i,j)\notin\Theta,\, (j,i)\in\Theta.
\end{align*}
Obviously, the Kiselman monoid $\mathcal{K}_n$ arises as a special case of this construction when the role of $\Theta$ is played by the strict ``descending'' order $\Theta_K=\{(j,i)\mid 1\le i < j \le n\}$. Another important special case is the so-called \emph{Catalan monoid}, denoted $\mathcal{C}_{n+1}$, that is generated by elements $a_1,a_2,\dots,a_n$ subject to the relations
\begin{align*}
&a_i^2=a_i                  &&\text{for each } i=1,\dots,n;\\
&a_{i}a_{j}=a_{j}a_{i}      &&\text{if }|i-j|\ge 2,\ i,j=1,\dots,n;\\
&a_ia_{i+1}a_i=a_{i+1}a_ia_{i+1}=a_{i+1}a_i &&\text{for each } i=1,\dots,n-1.
\end{align*}
Clearly, $\mathcal{C}_{n+1}$ is nothing but the Hecke--Kiselman monoid corresponding to the covering relation $\Theta_C=\{(i+1,i)\mid i=1,2,\dots,n-1\}$ of the order $\Theta_K$. Solomon~\cite{So96} has shown that $\mathcal{C}_{n+1}$ is isomorphic to the monoid of all \op\ and decreasing transformations of the chain
\[
1<2<\dots<n<n+1.
\]
(Recall that transformation $\alpha$ of a partially ordered set $\langle Q,\le\rangle$ is called \emph{\op} if $q\le q'$ implies $q\alpha\le q'\alpha$ for all $q,q'\in Q$, and \emph{decreasing} if $q\alpha\le q$ for every $q\in Q$.) Hence the monoid $\mathcal{C}_n$ is finite for every $n$; moreover, it can be shown that the cardinality of $\mathcal{C}_n$ is the $n$-th Catalan number $\frac1{n+1}\binom{2n}n$; see, e.g., \cite[Theorem~3.1]{Hi93}.

In~\cite{Vo04}, the second-named author has investigated the identities holding in each Catalan monoid and solved the \fbp\ for these identities. Here we extend this study to Kiselman monoids and, more generally, to all Hecke--Kiselman monoids $\mathcal{HK}_\Theta$ such that $\Theta_C\subseteq\Theta\subseteq\Theta_K$.

The paper is basically self-contained and is structured as follows. In Section~\ref{sec:preliminaries} we briefly discuss the finite basis problem in order to place the present study into a proper perspective. We also summarize a few properties of Catalan, Kiselman and Hecke--Kiselman monoids that are essential for the proof of our main result. This result is formulated and proved in Section~\ref{sec:new} while Section~\ref{sec:open problems} collects some additional results and a few related open questions.

\section{Preliminaries}
\label{sec:preliminaries}

A \emph{\sgp\ identity} is just a pair of \emph{words}, i.e., elements of the free semigroup $X^+$ over an alphabet $X$. In this paper identities are written as ``bumped'' equalities such as $u\bumpeq v$. A \sgp\ $S$ \emph{satisfies} $u\bumpeq v$ where $u,v\in X^+$ if for every homomorphism $\varphi\colon X^+\to S$, the equality $u\varphi=v\varphi$ is valid in $S$; alternatively, we say that $u\bumpeq v$ \emph{holds} in $S$.

Given any system $\Sigma$ of \sgp\ identities, we say that an identity $u\bumpeq v$ \emph{follows} from $\Sigma$ if every \sgp\ satisfying all identities of $\Sigma$ satisfies the identity $u\bumpeq v$ as well; alternatively, we say that $\Sigma$ \emph{implies} $u\bumpeq v$. A subset $\Sigma'\subseteq\Sigma$ is called an \emph{\ib} for $\Sigma$ if each identity in $\Sigma$ follows from $\Sigma'$. Given a \sgp\ $S$, its \emph{equational theory} is the set of all identities holding in $S$. A \sgp\ is said to be \emph{\fb} if its equational theory admits a finite \ib;  otherwise it is called \emph{\nfb}.

It was discovered by Perkins \cite{Pe66,Pe69} that a \fs\ can be \nfb. In fact, \sgps\ are the only ``classical'' algebras for which finite \nfb\ objects exist: finite groups \cite{OaPo64}, finite associative and Lie rings \cite{Kr73,Lv73,BaOl75}, finite lattices \cite{McK70} are all \fb. This circumstance gave rise to numerous
investigations whose final aim was to classify all \fss\ with respect to the property of having/having no finite \ib. In spite of many researchers' efforts, the complete classification has not yet been achieved, and therefore, it appears to be reasonable to look at some restricted versions of the problem where one focuses on certain important classes of \fss. This paper contributes to the study of the \fbp\ within the class of finite $\mathrsfs{J}$-trivial monoids.

Recall that a monoid $M$ is said to be $\mathrsfs{J}$-\emph{trivial} if every principal ideal of $M$ has a unique generator, that is, $MaM=MbM$ implies $a=b$ for all $a,b\in M$. Finite $\mathrsfs{J}$-trivial monoids attract much attention because of their distinguished role in algebraic language theory~\cite{Si72,Si75} and representation theory~\cite{DHST11}. The \fbp\ remains extremely hard when restricted to this class of monoids. In fact, one of the very first examples of \nfb\ \fss\ constructed by Perkins~\cite{Pe69} was a $\mathrsfs{J}$-trivial monoid, and further studies of the underlying construction have revealed a very complicated behaviour of $\mathrsfs{J}$-trivial monoids with respect to the finite basis property; see \cite[Subsection~4.2]{Vo01} for an overview of related results.

The Catalan monoids $\mathcal{C}_n$ defined in Section~\ref{sec:intro} are known to be $\mathrsfs{J}$-trivial, and moreover, they serve as sort of universal objects for the class of all finite $\mathrsfs{J}$-trivial monoids. Recall that a monoid $M$ is said to \emph{divide} a monoid $N$ if $M$ is a homomorphic image of a submonoid in $N$. Straubing~\cite{St80} has shown that a finite monoid is $\mathrsfs{J}$-trivial if and only if it divides some monoid of \op\ and decreasing transformations of a finite partially ordered set. Pin \cite[Theorem 4.1.10]{Pi86} has observed that this result can be made more concrete: a finite monoid is $\mathrsfs{J}$-trivial if and only if it divides some Catalan monoid $\mathcal{C}_n$. The \fbp\ for the monoids  $\mathcal{C}_n$ has been solved by the second-named author \cite{Vo04} on the basis of some results by Blanchet-Sadri~\cite{Bl93,Bl94}.
The following summarizes this solution.

\begin{theorem}
\label{thm:fbp for Cn}
\emph{a)} The identities
\begin{equation}
\label{eq:basis c3}
xyxzx\bumpeq xyzx,\ (xy)^2\bumpeq (yx)^2
\end{equation}
form an \ib\ of the monoid $\mathcal{C}_3$.

\emph{b)} The identities
\begin{gather}
xyx^2zx\bumpeq xyxzx,\ xyzx^2tz\bumpeq xyxzx^2tz,\ zyx^2ztx\bumpeq zyx^2zxtx,\notag\\
(xy)^3\bumpeq (yx)^3
\label{eq:basis c4}
\end{gather}
form an \ib\ of the monoid $\mathcal{C}_4$.

\emph{c)} The monoids $\mathcal{C}_n$ with $n\ge5$ are \nfb.
\end{theorem}

\begin{remarks}
1. For compactness, when presenting the identity bases in Theorem~\ref{thm:fbp for Cn}, we have taken the liberty to use bases in the sense of \emph{monoid} identities. For instance, we have not included in~\eqref{eq:basis c3} the identity $x^3\bumpeq x^2$ (which does hold in $\mathcal{C}_3$) because in the monoid setting it can be deduced from the identity $xyxzx\bumpeq xyzx$ by substituting 1 for $y$ and $z$. However, as observed in~\cite[p.173]{Vo01}, the property of a monoid to be \fb\ or \nfb\ does not depend on using the semigroup or the monoid deduction rules.

2. Recently, alternative (and quite transparent) proofs of the results by Blanchet-Sadri utilized in Theorem~\ref{thm:fbp for Cn} have been suggested by Olga Sapir, see \cite[Proposition 4.2 and Theorem 7.2]{Sa1} and \cite[Corollary~6.4]{Sa2}.
\end{remarks}

We also need a characterization of equational theories of Catalan monoids. Let $X^*$ stand for the free monoid over an alphabet $X$. A word $u=x_1\cdots x_k$ with $x_1,\dots,x_k\in X$ is a \emph{scattered subword} of $v\in X^*$ whenever there exist words $v_0,v_1,\dots,v_{k-1},v_k\in X^*$ such that $v=v_0x_1v_1\cdots v_{k-1}x_kv_k$; in other terms,  one can extract $u$ treated as a sequence of letters from the sequence $v$. We denote by $J_n$ the set of all identities $w\bumpeq w'$ such that the words $w$ and $w'$ have the same set of scattered subwords of length at most $n$. The following is a combination of~\cite[Proposition~4]{Vo04} and~\cite[Corollary~2]{Vo04}.

\begin{proposition}
\label{prop:identities of Cn}
The equational theory of the monoid $\mathcal{C}_{n+1}$ is equal to $J_n$.
\end{proposition}

Now we turn to Kiselman monoids. We represent elements of the Kiselman monoid $\mathcal{K}_n$ by words over the alphabet $A_n=\{a_1,a_2,\dots,a_n\}$. For a word $w$, its \emph{content} $c(w)$ is the set of all letters that occur in $w$ (in particular, the content of the empty word 1 is the empty set). The relations in~\eqref{eq:kn} are such that the same letters occur in both sides of each relation whence $c(w)=c(w')$ whenever $w,w'\in A_n^*$ represent the same element of~$\mathcal{K}_n$. Therefore one can speak of the content of an element from $\mathcal{K}_n$. The following reduction rules underlie all combinatorics of Kiselman monoids.

\begin{lemma}[{\mdseries\cite[Lemma~1]{KM09}}]
\label{lem:reduction}
\emph{a)} Let $s\in\mathcal{K}_n$ and $c(s)\subseteq\{a_{i+1},\dots,a_n\}$ for some $i$. Then $a_isa_i=sa_i$.

\emph{b)} Let $t\in\mathcal{K}_n$ and $c(t)\subseteq\{a_1,\dots,a_{i-1}\}$ for some $i$. Then $a_ita_i=a_it$.
\end{lemma}

For applications in the present paper we modify Lemma~\ref{lem:reduction} as follows.
\begin{lemma}
\label{lem:reduction1}
\emph{a)} Let $s\in\mathcal{K}_n$ and $c(s)\subseteq\{a_i,a_{i+1},\dots,a_n\}$ for some $i$. Then $sa_i=s'a_i$, where $s'$ is obtained from $s$ by removing all occurrences of $a_i$ if there were some.

\emph{b)} Let $t\in\mathcal{K}_n$ and $c(t)\subseteq\{a_1,\dots,a_{i-1},a_i\}$ for some $i$. Then $a_it=a_it'$, where $t'$ is obtained from $t$ by removing all occurrences of $a_i$ if there were some.
\end{lemma}

\begin{proof}
We verify only claim a); the argument for claim b) is symmetric.

We induct on the number of occurrences of $a_i$ in $s$. If $a_i$ does not occur in $s$, nothing is to prove. Suppose that $a_i$ occurs in $s$ and represent $s$ as $s=s_1a_is_2$ where $s_2$ contains no occurrences of $a_i$. Then $sa_i=s_1a_is_2a_i=s_1s_2a_i$ by Lemma~\ref{lem:reduction}a, and the induction hypothesis applies to $s_1s_2a_i$ since $s_1s_2$ contains fewer occurrences of $a_i$.
\end{proof}

Finally, we recall the following property of Hecke--Kiselman monoids.

\begin{proposition}[{\mdseries\cite[Proposition~14]{GM11}}]
\label{prop:epimorphism}
Let $\Theta$ and $\Phi$ be anti-reflexive binary relations on the set $\{1,2,\dots,n\}$ and $\Phi\subseteq\Theta$. The map $a_i\mapsto a_i$ uniquely extends to a homomorphism from the Hecke--Kiselman monoid $\mathcal{HK}_\Theta$ onto the monoid $\mathcal{HK}_\Phi$.
\end{proposition}

\section{Equational theories of Hecke--Kiselman monoids}
\label{sec:new}

For each positive integer $n$ and each alphabet $X$, consider the following relation $\sim_n$ on the free monoid $X^*$:
\[
w\sim_n w'\ \text{ iff $w$ and $w'$ have the same scattered subwords of length $\le n$.}
\]
It is convenient for us to introduce also the relation $\sim_0$ by which we merely mean the universal relation on $X^*$. We need a few facts from combinatorics on words dealing with the relation $\sim_n$. They all come from Simon's seminal paper~\cite{Si75} and are collected in the following lemma.

\begin{lemma}
\label{lem:simon}
Let $n\ge1$ and $u,v\in X^*$.

\emph{a)} If $u\sim_n v$, then there exists a word $w\in X^*$ that has each of the words $u$ and $v$ as a scattered subword and such that $u\sim_n w\sim_n v$.

\emph{b)} The relation $u\sim_n uw$ holds if and only if there exist $u_1,\dots,u_n\in X^*$ such that $u=u_n\cdots u_1$ and $c(u_n)\supseteq\dots\supseteq c(u_1)\supseteq c(w)$.

\emph{c)} The relation $v\sim_n wv$ holds if and only if there exist $v_1,\dots,v_n\in X^*$ such that $v=v_1\cdots v_n$ and $c(w)\subseteq c(v_1)\subseteq\dots\subseteq c(v_n)$.

\emph{d)} Let $x\in X$. The relation $uxv\sim_n uv$ holds if and only if there exist $k,\ell\ge0$ with $k+\ell\ge n$ and such that $u\sim_k ux$ and $xv\sim_\ell v$.
\end{lemma}

Recall that $J_n$ stands for the set of all identities $w\bumpeq w'$ such that $w$ and $w'$ have the same scattered subwords of length $\le n$; that is,
\[
J_n=\{w\bumpeq w' \mid w\sim_n w'\}.
\]
Lemma~\ref{lem:simon}a implies that the set
\[
J'_n=\{w\bumpeq w' \mid w\sim_n w'\ \text{and}\ w\ \text{is a scattered subword of}\ w'\}
\]
forms an \ib\ for $J_n$. Now let $w\bumpeq w'$ be an identity in $J'_n$. Since $w$ is  a scattered subword of $w'$, we have $w=x_1\cdots x_m$ and
\[
w'=v_0x_1v_1\cdots v_{m-1}x_mv_m
\]
for some letters $x_1,\dots,x_m$ and some $v_0,v_1,\dots,v_{m-1},v_m\in X^*$. If we represent the word $v_0v_1\cdots v_{m-1}v_m$
as a product of letters,
\[
v_0v_1\cdots v_{m-1}v_m=y_1y_2\cdots y_p\quad \text{with }\ y_1,y_2,\dots,y_p\in X,
\]
we can think of $w'$ as of the result of $p$ successive insertions of the letters $y_1,\dots,y_p$ in the word $w$. If we let $w_0=w$ and denote by $w_i$, $i=1,\dots,p$, the word obtained after the $i$-th insertion, then $w_p=w'$,
and since $w\sim_n w'$, we conclude that $w_{i-1}\sim_n w_i$ for each $i$. Therefore the identities $w_0\bumpeq w_1$, $w_1\bumpeq w_2$, \dots, $w_{p-1}\bumpeq w_p$ all belong to $J'_n$ and, clearly, together they imply the identity $w\bumpeq w'$. Thus, we can substitute the set $J'_n$ by the following smaller \ib\ for $J_n$:
\[
J''_n=\{w\bumpeq w' \mid w\sim_n w'\ \text{and}\ w=uv,\ w'=uxv\ \text{for some}\ u,v\in X^*,\ x\in X\}.
\]

Combining this observation with Lemma~\ref{lem:simon}b--d, we get the following.

\begin{corollary}
\label{cor:uxv}
For each $n\ge1$, the collection of all identities of the form
\begin{equation}
\label{eq:uxv}
u_k\cdots u_1v_1\cdots v_\ell\bumpeq u_k\cdots u_1xv_1\cdots v_\ell,
\end{equation}
where $k,\ell\ge0$, $k+\ell\ge n$, $x\in X$, and
\begin{equation}
\label{eq:inclusions}
c(u_k)\supseteq\dots\supseteq c(u_1)\supseteq\{x\}\subseteq c(v_1)\subseteq\dots\subseteq c(v_\ell)
\end{equation}
forms an \ib\ for the set $J_n$.
\end{corollary}
Let us comment on the meaning of the formulas~\eqref{eq:uxv} and~\eqref{eq:inclusions} in the extreme situations
$k=0$ or $\ell=0$. If $k=0$, the identity~\eqref{eq:uxv} becomes
$v_1\cdots v_\ell\bumpeq xv_1\cdots v_\ell$, while the inclusions~\eqref{eq:inclusions} reduce to
$\{x\}\subseteq c(v_1)\subseteq\dots\subseteq c(v_\ell)$; dually, if $\ell=0$, then~\eqref{eq:uxv} and~\eqref{eq:inclusions} become respectively $u_k\cdots u_1\bumpeq u_k\cdots u_1x$ and $c(u_k)\supseteq\dots\supseteq c(u_1)\supseteq\{x\}$.

Fix an integer $n\ge 2$. Recall that
\begin{gather*}
\Theta_K=\{(j,i)\mid 1\le i < j \le n\},\\
\Theta_C=\{(i+1,i)\mid i=1,2,\dots,n-1\}.
\end{gather*}
Our main result describes the identities holding in the Hecke--Kiselman monoid $\mathcal{HK}_\Theta$ for every relation $\Theta$ situated between $\Theta_K$ and $\Theta_C$.

\begin{theorem}
\label{thm:identities of HKn}
Let $n\ge 2$. Then for every relation $\Theta$ on the set $\{1,2,\dots,n\}$ such that $\Theta_C\subseteq\Theta\subseteq\Theta_K$, the set of identities of the Hecke--Kiselman monoid $\mathcal{HK}_\Theta$ coincides with $J_n=\{w\bumpeq w' \mid w\sim_n w'\}$.
\end{theorem}

\begin{proof}
Let $\Sigma_\Theta$ stand for the set of identities holding in the monoid $\mathcal{HK}_\Theta$. By Proposition~\ref{prop:epimorphism} the Catalan monoid $\mathcal{C}_{n+1}=\mathcal{HK}_{\Theta_C}$ is a homomorphic image of $\mathcal{HK}_\Theta$ whence every identity of $\Sigma_\Theta$ holds in $\mathcal{C}_{n+1}$. By Proposition~\ref{prop:identities of Cn} this means that $\Sigma_\Theta\subseteq J_n$. On the other hand, by Proposition~\ref{prop:epimorphism} the monoid $\mathcal{HK}_\Theta$ is a homomorphic image of the Kiselman monoid $\mathcal{K}_n=\mathcal{HK}_{\Theta_K}$ whence $\mathcal{HK}_\Theta$ satisfies all identities of $\mathcal{K}_n$.
Therefore, in order to prove that $J_n\subseteq\Sigma_\Theta$, it suffices to check that every identity in the set $J_n$ holds in $\mathcal{K}_n$. Corollary~\ref{cor:uxv} reduces the latter task to verifying that $\mathcal{K}_n$ satisfies each identity of the form~\eqref{eq:uxv} obeying the conditions~\eqref{eq:inclusions}.

Thus, we fix an arbitrary identity of the form~\eqref{eq:uxv} satisfying~\eqref{eq:inclusions} and an arbitrary homomorphism $\varphi\colon X^+\to\mathcal{K}_n$. Given a word $w\in X^+$, we write $\bar{w}$ instead of $w\varphi$ for the image $w$ under $\varphi$, just to lighten notation. Let
\[
b=\bar{u}_k\cdots \bar{u}_1\bar{v}_1\cdots\bar{v}_\ell,\quad
d=\bar{u}_k\cdots \bar{u}_1\bar{x}\bar{v}_1\cdots\bar{v}_\ell.
\]
We aim to prove that $b=d$.

Recall that we represent elements of $\mathcal{K}_n$ as words over the alphabet $A_n=\{a_1,a_2,\dots,a_n\}$, and it makes sense to speak about occurrences of a letter in an element $s\in\mathcal{K}_n$ since all words representing $s$ have the same content. If $p,q\in\{1,\dots,n\}$, we denote by $s^{[p,q]}$ the element of $\mathcal{K}_n$ obtained from $s$ by removing all occurrences of the letters $a_i$ such that either $i<p$ or $q<i$ if there were some. If no letter $a_i$ with either $i<p$ or $q<i$ occurs in $s$, we let $s^{[p,q]}=s$. Observe that, by this definition, $s^{[p,q]}=1$ whenever $p>q$.

Clearly, the inclusions~\eqref{eq:inclusions} imply that
\begin{equation}
\label{eq:inclusions1}
c(\bar{u}_k)\supseteq\dots\supseteq c(\bar{u}_1)\supseteq c(\bar{x})\subseteq c(\bar{v}_1)\subseteq\dots\subseteq c(\bar{v}_\ell).
\end{equation}

Suppose that $\ell>0$ and $a_1\in c(\bar{v}_\ell)$. Then we can apply Lemma~\ref{lem:reduction1}a to remove all occurrences of $a_1$ from each of the factors $\bar{v}_1,\dots,\bar{v}_{\ell-1}$ of $b$ and from each of the factors $\bar{x},\bar{v}_1,\dots,\bar{v}_{\ell-1}$ of $d$ without changing $b$ nor $d$. Using the notation introduced above, we get the equalities
\begin{align}
b&=\bar{u}_k\cdots\bar{u}_1\bar{v}_1^{[2,n]}\cdots\bar{v}_{\ell-1}^{[2,n]}\bar{v}_{\ell},\label{eq:1st step b}\\
d&=\bar{u}_k\cdots\bar{u}_1\bar{x}^{[2,n]}\bar{v}_1^{[2,n]}\cdots\bar{v}_{\ell-1}^{[2,n]}\bar{v}_{\ell}.
\label{eq:1st step d}
\end{align}
However, the inclusions~\eqref{eq:inclusions1} ensure that the equalities~\eqref{eq:1st step b} and~\eqref{eq:1st step d} hold true also in the case when $a_1\notin c(\bar{v}_\ell)$: if the letter $a_1$ does not occur in $\bar{v}_\ell$, it occurs in none of the factors $\bar{x},\bar{v}_1,\dots,\bar{v}_{\ell-1}$, whence $\bar{x}=\bar{x}^{[2,n]}$,
$\bar{v}_1=\bar{v}_1^{[2,n]}$, \dots, $\bar{v}_{\ell-1}=\bar{v}_{\ell-1}^{[2,n]}$. Thus, the equalities~\eqref{eq:1st step b} and~\eqref{eq:1st step d} hold whenever $\ell>0$.

Suppose that $\ell>1$ and $a_2\in c(\bar{v}_{\ell-1})$. Then we can apply Lemma~\ref{lem:reduction1}a to remove all occurrences of $a_2$ from each of the factors $\bar{v}_1^{[2,n]},\dots,\bar{v}_{\ell-1}^{[2,n]}$ in the representation~\eqref{eq:1st step b} of $b$ and from each of the factors $\bar{x}^{[2,n]},\bar{v}_1^{[2,n]},\dots,\bar{v}_{\ell-1}^{[2,n]}$ in the representation~\eqref{eq:1st step d} of $d$  so that neither $b$ nor $d$ will change. Thus, we get
\begin{align}
b&=\bar{u}_k\cdots\bar{u}_1\bar{v}_1^{[3,n]}\cdots\bar{v}_{\ell-2}^{[3,n]}\bar{v}_{\ell-1}^{[2,n]}\bar{v}_{\ell},
\label{eq:2nd step b}\\
d&=\bar{u}_k\cdots\bar{u}_1\bar{x}^{[3,n]}\bar{v}_1^{[3,n]}\cdots\bar{v}_{\ell-2}^{[3,n]}\bar{v}_{\ell-1}^{[2,n]}
\bar{v}_{\ell}.\label{eq:2nd step d}
\end{align}
Again, the inclusions~\eqref{eq:inclusions1} imply that the equalities~\eqref{eq:2nd step b} and~\eqref{eq:2nd step d} hold also in the case when $a_2\notin c(\bar{v}_{\ell-1})$.

Applying this argument $\ell-1$ times to $b$ and $\ell$ times to $d$, we eventually arrive at the equalities
\begin{align}
b&=\bar{u}_k\cdots\bar{u}_1\bar{v}_1^{[\ell,n]}\cdots\bar{v}_{\ell-2}^{[3,n]}\bar{v}_{\ell-1}^{[2,n]}\bar{v}_{\ell},
\label{eq:ell-th step b}\\
d&=\bar{u}_k\cdots\bar{u}_1\bar{x}^{[\ell+1,n]}\bar{v}_1^{[\ell,n]}\cdots\bar{v}_{\ell-2}^{[3,n]}\bar{v}_{\ell-1}^{[2,n]}
\bar{v}_{\ell}.\label{eq:ell-th step d}
\end{align}
Now we can apply the symmetric argument on the left. If $k>0$ and $a_n\in c(\bar{u}_k)$, we use Lemma~\ref{lem:reduction1}b to remove all occurrences of $a_n$ from each of the factors $\bar{u}_k,\dots,\bar{u}_1$ in the representation~\eqref{eq:ell-th step b} of $b$ and from each of the factors  $\bar{u}_k,\dots,\bar{u}_1,\bar{x}^{[\ell+1,n]}$ in the representation~\eqref{eq:ell-th step d} of $d$ with no effect on the value of $b$ or $d$. This leads to the equalities
\begin{align*}
b&=\bar{u}_k\bar{u}_{k-1}^{[1,n-1]}\cdots\bar{u}_1^{[1,n-1]}\bar{v}_1^{[\ell,n]}\cdots
\bar{v}_{\ell-2}^{[3,n]}\bar{v}_{\ell-1}^{[2,n]}\bar{v}_{\ell},\\
d&=\bar{u}_k\bar{u}_{k-1}^{[1,n-1]}\cdots\bar{u}_1^{[1,n-1]}\bar{x}^{[\ell+1,n-1]}\bar{v}_1^{[\ell,n]}\cdots
\bar{v}_{\ell-2}^{[3,n]}\bar{v}_{\ell-1}^{[2,n]}\bar{v}_{\ell},
\end{align*}
that hold also in the case when $a_n\notin c(\bar{u}_k)$. Applying the ``left'' argument $k-1$ times to $b$ and $k$ times to $d$, we finally arrive at the following equalities:
\begin{align}
b&=\bar{u}_k\bar{u}_{k-1}^{[1,n-1]}\bar{u}_{k-2}^{[1,n-2]}\cdots\bar{u}_1^{[1,n-k+1]}\bar{v}_1^{[\ell,n]}\cdots
\bar{v}_{\ell-2}^{[3,n]}\bar{v}_{\ell-1}^{[2,n]}\bar{v}_{\ell},\label{eq:(ell+k)-th step b}\\
d&=\bar{u}_k\bar{u}_{k-1}^{[1,n-1]}\bar{u}_{k-2}^{[1,n-2]}\cdots\bar{u}_1^{[1,n-k+1]}\bar{x}^{[\ell+1,n-k]}
\bar{v}_1^{[\ell,n]}\cdots\bar{v}_{\ell-2}^{[3,n]}\bar{v}_{\ell-1}^{[2,n]}\bar{v}_{\ell}.\label{eq:(ell+k)-th step d}
\end{align}
Since $k+\ell\ge n$, we have $\ell+1>n-k$ whence $\bar{x}^{[\ell+1,n-k]}=1$. Thus, from the representations \eqref{eq:(ell+k)-th step b} and \eqref{eq:(ell+k)-th step d} we conclude that $b=d$, as required.
\end{proof}

Combining Theorem~\ref{thm:identities of HKn}, Proposition~\ref{prop:identities of Cn}, and Theorem~\ref{thm:fbp for Cn}, we immediately obtain a solution to the \fbp\ for the Hecke--Kiselman monoids $\mathcal{HK}_\Theta$ with $\Theta$ satisfying $\Theta_C\subseteq\Theta\subseteq\Theta_K$.

\begin{corollary}
\label{cor:fbp for HKn} Let $n\ge 2$ and let $\Theta$ be a binary relation on the set $\{1,2,\dots,n\}$ satisfying $\Theta_C\subseteq\Theta\subseteq\Theta_K$.

\emph{a)} If $n=2$ or $n=3$, then the monoid $\mathcal{HK}_\Theta$ is \fb\ and has the identities~\eqref{eq:basis c3} or respectively~\eqref{eq:basis c4} as an \ib.

\emph{b)} If $n\ge 4$, then the monoid $\mathcal{HK}_\Theta$ is \nfb.
\end{corollary}

Specializing Corollary~\ref{cor:fbp for HKn} for $\Theta=\Theta_K$, we get a complete solution to the \fbp\ for the Kiselman monoids $\mathcal{K}_n$.

\begin{corollary}
\label{cor:fbp for Kn}
The monoids $\mathcal{K}_2$ and $\mathcal{K}_3$ are \fb\ while the monoids $\mathcal{K}_n$ with $n\ge 4$ are \nfb.
\end{corollary}

\begin{remark}
Even though, for the sake of completeness, we have included the case $n=2$ in the formulations of Theorem~\ref{thm:identities of HKn} and Corollaries~\ref{cor:fbp for HKn} and~\ref{cor:fbp for Kn}, none of these results are new for the special case. Indeed, on the set $\{1,2\}$, the relations $\Theta_C$ and $\Theta_K$ coincide, whence the Kiselman monoid $\mathcal{K}_2$ is nothing but the Catalan monoid $\mathcal{C}_3$ and so are all the Hecke--Kiselman monoids $\mathcal{HK}_\Theta$ such that $\Theta_C\subseteq\Theta\subseteq\Theta_K$.
\end{remark}

Finally, we present an example demonstrating that the above description of the equational theory of the Kiselman monoids may also be used to obtain information about their structural properties. The following result was announced in~\cite{Go03} while its first proof appeared in~\cite[Theorem 22]{KM09}; this proof uses a specific representation of $\mathcal{K}_n$ by integer $n\times n$-matrices and involves a clever estimation of an ad hoc numerical parameter of such matrices. In contrast, our proof is rather straightforward and works for all Hecke--Kiselman monoids $\mathcal{HK}_\Theta$ with $\Theta\subseteq\Theta_K$.

\begin{corollary}
\label{cor:Kn is J-trivial}
The Kiselman monoids $\mathcal{K}_n$ are $\mathrsfs{J}$-trivial.
\end{corollary}

\begin{proof*}
Lemma~\ref{lem:simon}b,c implies that $(xy)^n\sim_n(xy)^nx$ and $x(yx)^n\sim_n(yx)^n$. By Theorem~\ref{thm:identities of HKn}, the identities $(xy)^n\bumpeq(xy)^nx$ and $x(yx)^n\bumpeq(yx)^n$ hold in the monoid $\mathcal{K}_n$. It is well known that every monoid $M$ satisfying these two identities is $\mathrsfs{J}$-trivial but, for completeness, we reproduce an elementary proof of this fact.

Let $a,b\in M$ and $MaM=MbM$. Then $a=qbr$ and $b=sat$ for some $q,r,s,t\in M$. Substituting the second equality into the first one, we get $a=qsatr$ whence $a=(qs)^na(tr)^n$. Since $M$ satisfies $(xy)^n\bumpeq(xy)^nx$ and $x(yx)^n\bumpeq(yx)^n$, we have
\[
a=(qs)^na(tr)^n=s(qs)^na(tr)^nt=sat=b.\eqno{\qed}
\]
\end{proof*}

Similar syntactic arguments can be used to reprove some other structural results in~\cite{KM09}.

\section{Further results and open questions}
\label{sec:open problems}

\subsection{The \fbp\ for monoids $G(E)$}
\label{subsec:G(E)}
Corollary~\ref{cor:fbp for Kn} shows, in particular, that the ``classic'' Kiselman monoid $\mathcal{K}_3$ originated in \cite{Ki02} is \fb. Recall that $\mathcal{K}_3$ has 18 elements and arises as the monoid of the form $G(E)$ where the underlying space $E$ is a normed space of infinite dimension. If $E$ is finite-dimensional, then the order of the monoid $G(E)$ is always less than 18 and the value of the order varies, depending on the two parameters: the dimension of $E$ and the dimension of $\overline{\{0\}}$, the closure of the origin ($E$ is not assumed to be Hausdorff so that singletons need not be closed). Moreover, it follows from \cite[Theorem~4.1]{Ki02} that each monoid of the form $G(E)$ for finite-dimensional $E$ is isomorphic to the Rees quotient of $\mathcal{K}_3$ over a certain non-singleton ideal. We present a classification of the monoids in Table~\ref{tb:classification}, labelling them as in \cite[Theorem~4.1]{Ki02}.

\extrarowheight=4pt
\begin{center}
\begin{table}[ht]
\caption{The monoids $G(E)$ for finite-dimensional $E$}
\label{tb:classification}
\begin{tabular}{|c|p{2.9cm}|p{1cm}|c|}
\hline
\raisebox{-14pt}{Label} & Relation between $n{\,=\,}\dim E$ and\phantom{bl} $k=\dim\overline{\{0\}}$ & Order of $G(E)$ & \raisebox{-14pt}{The ideal $I$ with $\mathcal{K}_3/I\cong G(E)$}\\
\hline
A$_1$ & $n=0$ & \hfill 1  & $\mathcal{K}_3$\\
\hline
A$_{15}$ & $n=1$, $k=0$ & \hfill 15 & $\{a_2a_3a_1,a_2a_3a_1a_2,a_3a_1a_2,a_3a_2a_1\}$\\
\hline
A$_{16}$ & $n\ge2$, $k=0$ & \hfill 16 & $\{a_2a_3a_1,a_2a_3a_1a_2,a_3a_2a_1\}$\\
\hline
B$_6$ & $n=k>0$ & \hfill 6 & $\mathcal{K}_3a_2\mathcal{K}_3$\\
\hline
B$_{16}$ & $n-1=k>0$ & \hfill 16 & $\{a_2a_3a_1a_2,a_3a_1a_2,a_3a_2a_1\}$\\
\hline
B$_{17}$ & $n-2\ge k>0$ & \hfill 17 & $\{a_2a_3a_1a_2,a_3a_2a_1\}$\\
\hline
\end{tabular}
\end{table}
\end{center}

It may be worth commenting on the topology of $E$ in each of the non-trivial cases in Table~\ref{tb:classification}. In case A$_{15}$ the space $E$ is nothing but the real line $\mathbb{R}$ with the usual topology while case A$_{16}$ corresponds to $E=\mathbb{R}^n$, $n\ge2$, again with the usual topology. Case B$_6$ arises when $E$ is nonzero and equipped with the so-called \emph{chaotic} topology, i.e., the topology such that the only neighborhood of the origin is the whole space. In the remaining two cases (B$_{16}$ and B$_{17}$), one has a mixture between the usual and chaotic topologies: the space $E$ is isomorphic to the product of a $k$-dimensional chaotic space and the space $\mathbb{R}^{n-k}$ with the usual topology.

Being homomorphic images of $\mathcal{K}_3$, the monoids $G(E)$ satisfy the identities~\eqref{eq:basis c4} but it is known that these identities are not strong enough to ensure the finite basis property in every \fs\ satisfying them. Indeed, Lee (unpublished) has observed that the identities~\eqref{eq:basis c4} hold in the 6-element \sgp\ $\mathcal{L}$ given in the class of \sgps\ with 0 by the presentation
\begin{equation}
\label{eq:L}
\mathcal{L}=\langle e,f\mid e^2=e,\ f^2=f,\ efe=0\rangle,
\end{equation}
while the third-named author of the present paper and Luo~\cite{ZL11} have proved that this \sgp\ is \nfb. Nevertheless, here we show that each monoid of the form $G(E)$ with finite-dimensional $E$ is \fb.

Clearly, the 1-element monoid in case A$_1$ is \fb. As for the monoids in the five remaining rows of Table~\ref{tb:classification}, the \fbp\ for each of them is solved by the following somewhat surprising observation.

\begin{theorem}
\label{thm:G(E)}
Each of the monoids of the form $G(E)$ in cases \textup{A}$_{15}$, \textup{A}$_{16}$, \textup{B}$_{6}$, \textup{B}$_{16}$, and \textup{B}$_{17}$ has the same equational theory as the monoid $\mathcal{K}_2$, and therefore, has the system~\eqref{eq:basis c3} as its \ib.
\end{theorem}

\begin{proof}
Clearly, the submonoid $M$ of the Kiselman monoid $\mathcal{K}_3$ generated by $a_1$ and $a_2$ is isomorphic to the monoid $\mathcal{K}_2$. On the other hand, $M$ has empty intersection with each of the ideals $I$ corresponding to the non-singleton monoids of the form $G(E)$ whence $M$ embeds into $\mathcal{K}_3/I$. We conclude that $\mathcal{K}_2$ is isomorphic to a submonoid in each non-singleton monoid of the form $G(E)$, and therefore, all identities holding in any of the latter monoids hold also in $\mathcal{K}_2$.

The ideal $K=\{a_2a_3a_1a_2,a_3a_2a_1\}$ is contained in each of the ideals listed in the last column of  Table~\ref{tb:classification} whence each monoid of the form $G(E)$ is a homomorphic image of the monoid $\mathcal{K}_3/K$. Thus, in order to show that every identity holding in $\mathcal{K}_2$ holds also in each monoid of the form $G(E)$, it suffices to verify this for the monoid $\mathcal{K}_3/K$.

Recall that by Theorem~\ref{thm:identities of HKn} the set of identities of $\mathcal{K}_2$ coincides with the set $J_2=\{w\bumpeq w' \mid w\sim_2 w'\}$. By Corollary~\ref{cor:uxv}, the collection of all identities of the form~\eqref{eq:uxv} satisfying~\eqref{eq:inclusions} and such that $k+\ell\ge2$ forms an \ib\ for $J_2$. Thus, we take an arbitrary identity
\[
w=u_k\cdots u_1v_1\cdots v_\ell\bumpeq u_k\cdots u_1xv_1\cdots v_\ell=w'
\]
of this form and show that it holds in $\mathcal{K}_3/K$.

Let $\varphi\colon X^+\to\mathcal{K}_3/K$ be an arbitrary homomorphism. We adopt the notational conventions of the proof of Theorem~\ref{thm:identities of HKn}; in particular, we write $\bar{u}$ instead of $u\varphi$ for $u\in X^+$. Let
\[
b=\bar{u}_k\cdots \bar{u}_1\bar{v}_1\cdots\bar{v}_\ell,\quad
d=\bar{u}_k\cdots \bar{u}_1\bar{x}\bar{v}_1\cdots\bar{v}_\ell.
\]
We aim to show that $b=d$. Clearly, the equality holds provided that $\varphi$ maps some letter from $c(w)=c(w')$ to the zero of the monoid $\mathcal{K}_3/K$. Therefore, we may additionally assume that $\bar{y}\in\mathcal{K}_3\setminus  K$ for every letter $y\in c(w)=c(w')$; in other words, we may assume that the homomorphism $\varphi$ takes its values in $\mathcal{K}_3$ and modify our aim as follows: to show that either $b=d$ or $b,d\in K$.

First consider the case when $c(\bar{u}_k\bar{v}_\ell)\ne\{a_1,a_2,a_3\}$. Taking into account the inclusions~\eqref{eq:inclusions1} that follow from~\eqref{eq:inclusions}, we conclude that in this case $\varphi$ sends all the letters occurring in $w$ and $w'$ into a submonoid of $\mathcal{K}_3$ generated by two of the generators $a_1,a_2,a_3$. Each such submonoid is easily seen to be isomorphic to the monoid $\mathcal{K}_2$, and the latter monoid satisfies all identities in $J_2$ by Theorem~\ref{thm:identities of HKn}. Hence $b=d$ in this case.

Now assume that $c(\bar{u}_k\bar{v}_\ell)=\{a_1,a_2,a_3\}$. Here we subdivide our analysis into three subcases.

\smallskip

\emph{\textbf{Subcase 1:}} $k=0$, $\ell\ge 2$. In this subcase we have
\[
b=\bar{v}_1\cdots\bar{v}_\ell,\quad
d=\bar{x}\bar{v}_1\cdots\bar{v}_\ell,\quad
\text{and}\quad c(\bar{v}_\ell)=\{a_1,a_2,a_3\}.
\]
Applying Lemma~\ref{lem:reduction1}a as in the proof of Theorem~\ref{thm:identities of HKn}, we can rewrite $b$ and $d$ as follows:
\[
b=\bar{v}_1^{[\ell,3]}\cdots\bar{v}_{\ell-1}^{[2,3]}\bar{v}_{\ell},\quad
d=\bar{x}^{[\ell+1,3]}\bar{v}_1^{[\ell,3]}\cdots\bar{v}_{\ell-1}^{[2,3]}\bar{v}_{\ell}.
\]
If $\ell>2$, then $\bar{x}^{[\ell+1,3]}=1$ and $b=d$. The same conclusions follow provided that  $a_3\notin c(\bar{x})$. Hence we may assume that $\ell=2$ and $a_3\in c(\bar{x})$. Then $\bar{x}^{[\ell+1,3]}=\bar{x}^{[3,3]}$ must be a power of $a_3$ and, since $a_3^2=a_3$, we conclude that
\[
b=\bar{v}_1^{[2,3]}\bar{v}_2,\quad
d=a_3\bar{v}_1^{[2,3]}\bar{v}_2=a_3b.
\]
Clearly, Lemma~\ref{lem:reduction1} allows us to retain in $b$ only the leftmost occurrence of $a_3$ and only the rightmost occurrence of $a_1$. Since $a_1$ occurs in $\bar{v}_2$ and, in view of~\eqref{eq:inclusions1}, $a_3$ occurs in $\bar{v}_1^{[2,3]}$, the leftmost occurrence of $a_3$ precedes the rightmost occurrence of $a_1$ in $b$. We see that $b=b_1a_3b_2a_1b_3$ where each of $b_1,b_2,b_3$ is either 1 or a power of $a_2$, see Fig.~\ref{fig:Case 1}, in which $s$ stands for the element obtained from $\bar{v}_1^{[2,3]}$ by removing all occurrences of $a_3$ except the leftmost one, while $t$ denotes the element obtained from  $\bar{v}_2$ by removing all occurrences of $a_1$ except the rightmost one and all occurrences of $a_3$.
\begin{figure}[ht]
\centering
\unitlength=.7mm
\begin{picture}(170,17.5)(0,0)
\gasset{AHnb=0}
\drawline[linewidth=.5](5,7.5)(165,7.5)
\drawline(5,0)(5,15)
\drawline(165,0)(165,15)
\drawline(65,7.5)(65,15)
\drawline(37,7.5)(37,0)
\drawline(45,7.5)(45,0)
\drawline(85,7.5)(85,0)
\drawline(93,7.5)(93,0)
\gasset{AHnb=1}
\drawline(35,12.5)(65,12.5)
\drawline(35,12.5)(5,12.5)
\drawline(115,12.5)(65,12.5)
\drawline(115,12.5)(165,12.5)
\drawline(65,2.5)(45,2.5)
\drawline(65,2.5)(85,2.5)
\drawline(21,2.5)(5,2.5)
\drawline(21,2.5)(37,2.5)
\drawline(129,2.5)(165,2.5)
\drawline(129,2.5)(93,2.5)
\put(36,15){$s$}
\put(114,15){$t$}
\put(20,-3){$b_1$}
\put(64,-3){$b_2$}
\put(128,-3){$b_3$}
\put(38.5,1.5){$a_3$}
\put(86.5,1.5){$a_1$}
\end{picture}
\caption{Structure of $b$ in Subcase 1}\label{fig:Case 1}
\end{figure}
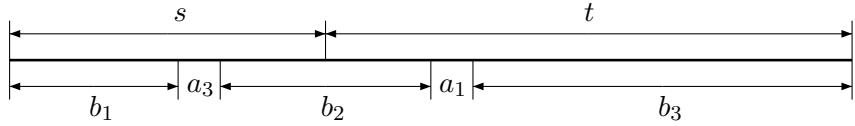

If $b_1=1$, we have $d=a_3b=a_3^2b_2a_1b_3=a_3b_2a_1b_3=b$. If $b_2\ne 1$, then taking into account that $a_2^2=a_2$, we obtain that $b$ has the product $a_3a_2a_1$ as a factor. If $b_1\ne 1$ while $b_2=1$, we observe that $b_3\ne 1$ since $a_2$ occurs in $\bar{v}_2$ and thus in $t$. Therefore, $b$ has the product $a_2a_3a_1a_2$ as a factor. Both $a_3a_2a_1$ and $a_2a_3a_1a_2$ belong to the ideal $K$ whence in either of the two remaining cases we have $b\in K$ and $d=a_3b\in K$.

\smallskip

\emph{\textbf{Subcase 2:}} $k\ge 2$, $\ell=0$. This subcase is dual to Subcase 1.

\smallskip

\emph{\textbf{Subcase 3:}} $k,\ell\ge 1$. Applying Lemma~\ref{lem:reduction1} as in the proof of Theorem~\ref{thm:identities of HKn}, we can rewrite $b$ and $d$ as follows:
\begin{align*}
b&=\bar{u}_k\bar{u}_{k-1}^{[1,2]}\cdots\bar{u}_1^{[1,3-k+1]}\bar{v}_1^{[\ell,3]}\cdots
\bar{v}_{\ell-1}^{[2,3]}\bar{v}_{\ell},\\
d&=\bar{u}_k\bar{u}_{k-1}^{[1,2]}\cdots\bar{u}_1^{[1,3-k+1]}\bar{x}^{[\ell+1,3-k]}
\bar{v}_1^{[\ell,3]}\cdots\bar{v}_{\ell-1}^{[2,3]}\bar{v}_{\ell}.
\end{align*}
If $k+\ell>2$, then $\bar{x}^{[\ell+1,3-k]}=1$ and $b=d$. The same conclusions follow if $a_2\notin c(\bar{x})$. Hence we may assume that $k=\ell=1$ and $a_2\in c(\bar{x})$. Then $\bar{x}^{[\ell+1,3-k]}=\bar{x}^{[2,2]}$ is a power of $a_2$ and, since $a_2^2=a_2$, we conclude that
\[
b=\bar{u}_1\bar{v}_1,\quad d=\bar{u}_1a_2\bar{v}_1.
\]
Now we apply Lemma~\ref{lem:reduction1} to $b$ and $d$ to eliminate all occurrences of $a_1$ except the rightmost one and all occurrences of $a_3$ except the leftmost one. If after that either $\bar{u}_1$ ends with $a_2$ or $\bar{v}_1$ starts with $a_2$, then the desired equality $b=d$ follows from $a_2^2=a_2$. In the remaining cases either
\[
b=b_1a_1a_3b_2,\quad d=b_1a_1a_2a_3b_2
\]
or
\[
b=b_1a_3a_1b_2,\quad d=b_1a_3a_2a_1b_2,
\]
where $b_1=a_2^p$ and $b_2=a_2^q$ for some $p,q>0$  because $a_2$ occurs in both $\bar{u}_1$ and $\bar{v}_1$ according to~\eqref{eq:inclusions1}. In the former case $d=a_2^pa_1a_2a_3b_2=a_2^pa_1a_3b_2=b$ in view of the relation $a_2a_1a_2=a_2a_1$ holding in $\mathcal{K}_3$. In the latter case, $b$ has the product $a_2a_3a_1a_2$ as a factor while $d$ has the product $a_3a_2a_1$ as a factor. Since both $a_3a_2a_1$ and $a_2a_3a_1a_2$ belong to the ideal $K$, we conclude that $b,d\in K$.
\end{proof}

\begin{remark}
As an alternative to the above proof, one could directly verify (using a computer) that the monoid $\mathcal{K}_3/K$ satisfies the identities~\eqref{eq:basis c3} and then refer to Theorem~\ref{thm:fbp for Cn}a. In fact, we have performed such a computation but preferred to include ``manual'' proof that uses only basic combinatorics of the monoid $\mathcal{K}_3/K$.
\end{remark}

\subsection{The \fbp\ for general Hecke--Kiselman monoids}
\label{subsec:HK}
Theorem~\ref{thm:identities of HKn} describes the equational theory of the Hecke--Kiselman monoids $\mathcal{HK}_\Theta$ such that the anti-reflexive binary relation $\Theta$ satisfies $\Theta_C\subseteq\Theta\subseteq\Theta_K$, and Corollary~\ref{cor:fbp for HKn} solves the \fbp\ problem for these monoids. However, questions of the same sort are of interest for an arbitrary $\Theta$. Here we discuss two open problems related to the equational theory of Hecke--Kiselman monoids.

Let $\Theta$ be an anti-reflexive binary relation on the set $V_n=\{1,2,\dots,n\}$. It is known~\cite[Theorem~16]{GM11} that the Hecke--Kiselman monoid $\mathcal{HK}_\Theta$ determines the graph $(V_n,\Theta)$ up to isomorphism, that is, for an arbitrary anti-reflexive binary relation $\Phi$ on the set $V_m=\{1,2,\dots,m\}$, the monoids  $\mathcal{HK}_\Theta$ and $\mathcal{HK}_\Phi$ are isomorphic if and only if the graphs $(V_n,\Theta)$ and $(V_m,\Phi)$ are isomorphic. On the other hand, Theorem~\ref{thm:identities of HKn} reveals that Hecke--Kiselman monoids with non-isomorphic underlying graphs can be \emph{equationally equivalent}, i.e., can have the same equational theory. For instance, if $n\ge 3$, then the Catalan monoid $\mathcal{C}_{n+1}=\mathcal{HK}_{\Theta_C}$ and the Kiselman monoid $\mathcal{K}_n=\mathcal{HK}_{\Theta_K}$ are non-isomorphic but these monoids are equationally equivalent by Theorem~\ref{thm:identities of HKn}. This observation gives rise to the following question.

\begin{question}
\label{que:ee for HK}
Let $\Theta$ and $\Phi$ be anti-reflexive binary relations on the sets $V_n=\{1,2,\dots,n\}$ and respectively $V_m=\{1,2,\dots,m\}$. Under which necessary and sufficient conditions on the graphs $(V_n,\Theta)$ and $(V_m,\Phi)$ are the Hecke--Kiselman monoids $\mathcal{HK}_\Theta$ and $\mathcal{HK}_\Phi$ equationally equivalent?
\end{question}

Similarly, with respect to the \fbp, the next question appears to be quite natural.

\begin{question}
\label{que:fbp for HK} Let $\Theta$ be an anti-reflexive binary
relation on the set $V_n=\{1,2,\dots,n\}$. Under which necessary
and sufficient conditions on the graph $(V_n,\Theta)$ is the
Hecke--Kiselman monoid $\mathcal{HK}_\Theta$ \fb?
\end{question}

Questions~\ref{que:ee for HK} and~\ref{que:fbp for HK} relate the study of Hecke--Kiselman monoids to the promising  area of investigations whose aim is to interpret graph-theoretical properties within equational properties of semigroups. Such an interpretation may shed new light on complexity-theoretical aspects of the theory of \sgp\ identities as a whole and of the \fbp\ in particular; see \cite{JM06} for an impressive instance of this approach. It is to expect, however, that Questions~\ref{que:ee for HK} and~\ref{que:fbp for HK} may be rather hard---for comparison, recall that in general it is still unknown for which graphs $(V_n,\Theta)$ the Hecke--Kiselman monoid $\mathcal{HK}_\Theta$ is finite, and even for $n=4$, the classification of graphs with finite Hecke--Kiselman monoids has turned out to be non-trivial, see~\cite{AD13}.

\subsection{The \fbp\ for involuted Kiselman monoids}
\label{subsec:involution}

A \sgp\ $S$ is called an \emph{\is} if it admits a unary operation
$s\mapsto s^*$ (called \emph{involution}) such that
$(st)^*=t^*s^*$ and $(s^*)^*=s$ for all $s,t\in S$. Whenever $S$
is equipped with a natural involution, it appears to be reasonable
to investigate the identities of $S$ as an algebra of type (2,1);
see~\cite{ADPV14,ADV12a,ADV12b,Lee_F} for numerous examples of recent studies
along this line. Observe that adding an involution may radically
change the equational properties of a \sgp: a \fb\ \sgp\ may
become a \nfb\ \is\ and vice versa. Infinite examples of this sort
have been known since 1970s (see \cite[Section~2]{Vo01} for
references and a discussion); more recently, Jackson and the
second-named author~\cite{JV10} have constructed a \fb\ \fs\ that
becomes a \nfb\ \is\ after adding a natural involution, while
Lee~\cite{Lee_Q} has shown that the 6-element \nfb\ \sgp\
$\mathcal{L}$ defined by~\eqref{eq:L} admits an involution under
which it becomes a \fb\ \is.

Inspecting the relations~\eqref{eq:kn}, one can easily see that the map $a_i\mapsto a_{n-i+1}$ uniquely extends to an involution of the Kiselman monoid $\mathcal{K}_n$; in fact, this is the only anti-automorphism of $\mathcal{K}_n$ \cite[Proposition 20b]{KM09}. In the same way, this map extends to an involution of the Catalan monoid $\mathcal{C}_{n+1}$. To the best of our knowledge, the equational properties of Kiselman and Catalan monoids treated as \iss\ have not been considered so far, and the examples mentioned in the preceding paragraph indicate that these properties need not necessarily follow the patterns revealed by the results of Section~\ref{sec:new}. Thus, we conclude with a pair of interrelated questions.

\begin{question}
\label{que:ee for inv}
Are the Kiselman monoid $\mathcal{K}_n$ and the Catalan monoid $\mathcal{C}_{n+1}$ equationally equivalent as \iss?
\end{question}

Obviously, the epimorphism $\mathcal{K}_n\to\mathcal{C}_{n+1}$ constructed according to Proposition~\ref{prop:epimorphism} is in fact a homomorphism of \iss\ so that $\mathcal{C}_{n+1}$ satisfies all \is\ identities that hold in $\mathcal{K}_n$. The converse, however, is very far from being clear.

\begin{question}
\label{que:fbp for inv}
a) For which $n$ is the Kiselman monoid $\mathcal{K}_n$ \fb\ as an \is?

b) For which $n$ is the Catalan monoid $\mathcal{C}_n$ \fb\ as an \is?
\end{question}

\subsection*{Acknowledgment} The authors are grateful to Edmond W. H. Lee for useful discussions.

\medskip

\paragraph*{\emph{Added in proof}} Very recently, Lee~\cite{Lee_U} has found a sufficient condition under which a \nfb\ \sgp\ remains \nfb\ when equipped with an involution. It is easy to verify that the condition holds true for Kiselman and Catalan monoids treated as \iss, and therefore, combining the results of~\cite{Lee_U} with those of~\cite{Vo04} and of the present paper, one immediately deduces the following partial answer to Question~\ref{que:fbp for inv}: the Kiselman monoid $\mathcal{K}_n$ \nfb\ as an \is\ for each $n\ge4$ and the Catalan monoid $\mathcal{C}_n$ \nfb\ as an \is\ for each $n\ge5$.

\end{document}